\documentclass{amsproc}
\usepackage{amsmath,amssymb,latexsym,amsfonts,cite,hyperref}
\usepackage[all]{xy}
\usepackage{amsthm}
\usepackage{enumerate}
\usepackage{color}
\usepackage{graphicx}
\usepackage{fullpage}

\hypersetup{unicode=false, pdftoolbar=true, pdfmenubar=true, pdffitwindow=false, pdfstartview={FitH}, pdfnewwindow=true, colorlinks=false, linkcolor=blue, citecolor=green}

\renewcommand{\O}{\mathcal O}

\newtheorem{theorem}{Theorem}[section]
\newtheorem{definition/proposition}[theorem]{Definition/Proposition}
\newtheorem{definition}[theorem]{Definition}
\newtheorem{proposition}[theorem]{Proposition}
\newtheorem{conjecture}[theorem]{Conjecture}
\newtheorem{corollary}[theorem]{Corollary}

\newtheorem{lemma}[theorem]{Lemma}
\numberwithin{equation}{section}
\theoremstyle{definition}

\title[Quantum Affine Schubert Cells]{Quantum Affine Schubert Cells and FRT-Bialgebras: The $E_6^{(1)}$ Case}
\author{Garrett Johnson and Christopher Nowlin}
\email{gwjohns3@ncsu.edu\\ cnowlin@math.ucsb.edu}

\begin{document}

\begin{abstract} De Concini, Kac, and Procesi defined a family of subalgebras ${\mathcal U}_q^+[w]\subseteq{\mathcal U}_q(\mathfrak{g})$ associated to elements $w$ in the Weyl group of a simple Lie algebra $\mathfrak{g}$. These algebras are called quantum Schubert cell algebras. We show that, up to a mild cocycle twist, quotients of certain quantum Schubert cell algebras of types $E_6$ and $E_6^{(1)}$ map isomorphically onto distinguished subalgebras of the Faddeev-Reshetikhin-Takhtajan universal bialgebra associated to the braiding on the quantum half-spin representation of ${\mathcal U}_q(\mathfrak{so}_{10})$. We identify the quotients as those obtained by factoring out the quantum Schubert cell algebras by ideals generated by certain submodules  with respect to the adjoint action of ${\mathcal U}_q(\mathfrak{so}_{10})$.
 \end{abstract}

\maketitle

\section{Introduction}

In a seminal ICM address, Drinfeld \cite{Dr} introduced the term ``quantum group'' and sparked a flurry of research which continues to this day.  Quantum groups are sometimes defined as Hopf algebras which are neither commutative nor cocommutative, but the phrase has been used loosely to describe a large class of related algebras which can be thought of as deformations of classical algebraic structures.  Two important classes of quantum groups are quantized enveloping algebras and quantized coordinate rings, considered as deformations of universal enveloping algebras of Lie algebras and coordinate rings of algebraic groups, respectively.

The original path to constructing the quantized coordinate ring $\mathcal{O}_q(G)$ for a reductive algebraic group $G$ is through a universal bialgebra construction due to Faddeev, Reshetikhin, and Takhtajan  \cite{FRT88}.  Let $\mathfrak{g}=\operatorname{Lie}(G)$ and denote by $\mathcal{U}_q(\mathfrak{g})$ the corresponding quantized enveloping algebra, where $q$ is a nonzero deformation parameter in the base field and not a root of unity.  To each finite-dimensional representation $V$ of ${\mathcal U}_q(\mathfrak{g})$, we associate a universal $R$-matrix $\widehat{R}$, a linear endomorphism of $V\otimes V$ satisfying the quantum Yang-Baxter equation.  There is a unique bialgebra $\mathcal A$ with $V$ as a comodule, universal with respect to the property that $\widehat{R}:V\otimes V\to V\otimes V$ is a map of ${\mathcal A}$-comodules. The bialgebra ${\mathcal A}$ has a standard presentation: if $\{v_1,...,v_n\}$ is a basis of $V$, then the standard generators of ${\mathcal A}$ are $X_{ij}$ for $i,j=1,...,n$. Thus, we colloquially refer to $X_{ij}$ as the generator in the $i$-th row and $j$-th column. The bialgebra ${\mathcal A}$ is graded with $\text{deg}(X_{ij})\addtolength{\textheight}{1.75in}=1$ for all $i,j=1,..,n$ and the defining relations are all quadratic, i.e. homogeneous of degree two. The FRT-bialgebra $\mathcal A$ is an intermediate step toward constructing $\O_q(G)$, which is realizable as a quotient of $\mathcal A$.  The most famous and well-studied examples of FRT-bialgebras are the algebras of $n\times n$ quantum matrices, denoted $\mathcal{O}_q(M_n)$.  Setting the quantum determinant of ${\mathcal O}_q(M_n)$ equal to $1\in{\mathcal O}_q(M_n)$ yields the Hopf algebra ${\mathcal O}_q(SL_n)$ (for details see, e.g., \cite[\textsection  9.1-9.2]{KS97}).

Our goal is to relate the FRT-bialgebras to quantum Schubert cells, distinguished subalgebras of quantized enveloping algebras first studied by De Concini, Kac, and Procesi \cite{DKP}. To each element $w$ in the Weyl group of a simple Lie algebra $\mathfrak{g}$, there is a standard construction that produces a set of quantum root vectors in $U_q(\mathfrak{g})$ which forms a Poincare-Birkhoff-Witt (PBW) basis for a quantum Schubert cell subalgebra ${\mathcal U}_q^+[w]$ (see, e.g., \cite[I.6.7]{BG}, \cite[\textsection 9.1.B]{CP}, or \cite[Ch. 37]{L}).  The algebra ${\mathcal U}_q^+[w]$ is a deformation of the universal enveloping algebra ${\mathcal U}(\mathfrak{n}_+\cap w\mathfrak{n}_-)$, where $\mathfrak{n}_{\pm}$ are a pair of opposite nilpotent subalgebras of $\mathfrak{g}$. For this reason, quantum Schubert cells have also been referred as to quantizations of nilpotent Lie algebras. From another perspective,  ${\mathcal U}_q^+[w]$ is a quantization of the coordinate ring of the Schubert cell $B_+w\cdot B_+$ in the full flag variety $G/B_+$ equipped with the standard Poisson structure \cite{GY}.  These algebras have also played an important role in recent years in the context of quantum cluster algebras \cite{GLS} and braided symmetric algebras \cite{BZ},\cite{Z}.

From a ring-theoretic perspective, quantum Schubert cell algebras have been of particular interest in an effort to develop a more general framework for studying iterated skew-polynomial rings. The algebras ${\mathcal U}_q^+[w]$ are noetherian and admit natural rational actions of algebraic tori $H$ by algebra automorphisms. Hence, the prime spectra of ${\mathcal U}_q^+[w]$ are partitioned into Goodearl-Letzter strata indexed by $H$-prime ideals. Each stratum is isomorphic to the prime spectrum of a Laurent polynomial ring \cite{GL}.  A description of the $H$-prime ideals of ${\mathcal U}_q^+[w]$ is given in \cite{Y}, where it was shown that the posets of $H$-primes, ordered under inclusion, are isomorphic to the Bruhat intervals $W^{\leq w}$. Finite generating sets for the $H$-prime ideals are also given in \cite{Y} in terms of Demazure modules. Formulas for the dimensions of the corresponding Goodearl-Letzter strata are given in \cite{BKL}, \cite{Y3}. 

In this paper, quantum Schubert cell algebras that have been twisted by a $2$-cocycle play an important role in the main results of Section \ref{section_mainresults}. Cocycle-twisted (or multiparameter) quantum Schubert cell algebras are included in a large family of algebras that extend the Artin-Schelter-Tate algebras of multiparameter quantum matrices \cite{AST91}. Several ring-theoretic results concerning these algebras have been obtained in \cite{Y2}. For instance, in \cite{Y2} a description of the prime spectra of multiparameter quantum Schubert cell algebras is given as well as formulas for the dimensions of the corresponding Goodearl-Letzter strata.

This paper can be considered a sequel to the authors' previous paper \cite{JN} where it was shown that, up to a mild cocycle twist, certain quantum affine Schubert cell algebras of types $A_n^{(1)}$ and $D_n^{(1)}$, or quotients thereof, map isomorphically onto certain distinguished subalgebras of types $A_{n-1}$ and $D_{n-1}$ FRT-bialgebras respectively.  In this paper, we turn our attention to the smallest of the exceptional types having a cominiscule root, namely the $E_6^{(1)}$ case, and prove an analogous result. In a forthcoming paper, we will cover the $E_7^{(1)}$ case.

To recall the details in \cite{JN}, we start with an extremal cominiscule root $\alpha$ in the root systems of types $A_n$ or $D_n$. We then consider the parabolic elements $w$ in the associated Weyl groups $W$ such that the sets of radical roots of $w$ are precisely $\{\beta\in\Phi_+ :\beta-\alpha\in Q_+\}$, where $\Phi_+$ is the full set of positive roots of type $A_n$ or $D_n$, and $Q_+=\mathbb{Z}_{\geq 0}\Phi_+$ is the corresponding positive root lattice. In these examples the corresponding quantum Schubert cell algebras ${\mathcal U}_q^+[w]$ are isomorphic to the well-studied quantum algebras known respectively as the $n$-dimensional quantum plane in the type $A_n$ case and even-dimensional quantum Euclidean space for the type $D_n$ case. Since the algebras ${\mathcal U}_q^+[w]$ are constructed from a cominiscule root, it follows that they are deformations of abelian nilradicals. Hence, the defining relations of ${\mathcal U}_q^+[w]$ are quadratic. Setting the deformation parameter $q$ equal to $1$ in these examples yields polynomial rings in commuting variables.

Next, we associate to $w$ an element $\widehat{w}$ in the corresponding affine Weyl group $\widehat W$ so that the set of radical roots of $\widehat{w}$ is precisely $\{\beta :\beta-\alpha\in Q_+\}\cup\{\beta +\delta : \beta-\alpha\in Q_+\}$, where $\delta$ is the null root in the corresponding affine root system. In \cite[Theorems 4.4 and 6.4]{JN}, we prove that the quantum affine Schubert cell algebras ${\mathcal U}_q^+[\widehat{w}]$ can be constructed via a quantized version of the semi-direct product Levi decomposition $\mathfrak{p} = \mathfrak{l}\ltimes rad(\mathfrak{p})$, where $\mathfrak{p}$ is the maximal parabolic subalgebra of $\mathfrak{sl}_{n+1}$ or $\mathfrak{so}_{2n}$ obtained by deleting the negative root $-\alpha$.

Finally, we let ${\mathcal A}$ denote the FRT-bialgebra associated to the braiding on the vector representation of ${\mathcal U}_q([\mathfrak{l},\mathfrak{l}])$, and let $T$ denote the subalgebra of ${\mathcal A}$ generated by a certain pair of ``rows" (for details, see the statement before Propostion 5.1 in \cite{JN}). Let $\left({\mathcal U}_q^+[\widehat{w}]\right)^\prime$ denote the algebra obtained from ${\mathcal U}_q^+[\widehat{w}]$ by twisting by a cocycle (for an explicit formula for the cocycle used, see the paragraph before Thm 5.2 in \cite{JN}). The cocycle-twisted algebra $\left({\mathcal U}_q^+[\widehat{w}]\right)^\prime$ is almost exactly like its untwisted counterpart except an extra power of $q$ is inserted into some of its defining relations. In the type $D_n^{(1)}$ case, we showed that there is a surjective algebra homomorphism $\Psi:\left({\mathcal U}_q^+[\widehat{w}]\right)^\prime\to T$ \cite[Prop. 5.1]{JN}. Furthermore, we identity the kernel of $\Psi$ as the ideal generated by elements which we labeled $\Omega_1$, $\Omega_2$, and $\Upsilon$. Although not explicitly mentioned, $\Omega_1$, $\Omega_2$, and $\Upsilon$ are certain invariants with respect to the adjoint action of ${\mathcal U}_q([\mathfrak{l},\mathfrak{l}])$ on ${\mathcal U}_q^+[\widehat{w}]$. For the type $A_n^{(1)}$ case, we simply have the cocycle-twisted quantum affine Schubert cell algebras mapping isomorphically onto $T$ \cite[Thm. 6.5]{JN}.

At present, we turn our attention to the type $E_6^{(1)}$ case.  Let $\mathfrak{e}_6$ denote the Lie algebra of exceptional type $E_6$, and let ${\bf I}$ be an index set of simple roots of $\mathfrak{e}_6$ (refer to Figure \ref{Dynkin diagram} for the numbering used throughout this paper). Up to an automorphism of the Dynkin diagram of $\mathfrak{e}_6$, there is one cominiscule root $\alpha_1$. The parabolic subalgebra $\mathfrak{p}\subseteq\mathfrak{e}_6$ associated to $\alpha_1$ decomposes $\mathfrak{p}=\mathfrak{l}\ltimes rad(\mathfrak{p})$ into a Levi subalgebra $\mathfrak{l}$ and a $16$-dimensional abelian nilradical $rad(\mathfrak{p})$.  We let ${\mathcal B}$ denote the set of even cardinality subsets of $\{1,...,5\}$ and index the quantum root vectors of ${\mathcal U}_q^+[w]$ naturally by elements of ${\mathcal B}$. For instance, the highest root vector gets the label $Y_\emptyset\in{\mathcal U}_q^+[w]$. The algebra ${\mathcal U}_q^+[w]$ is a left module algebra over ${\mathcal U}_q([\mathfrak{l},\mathfrak{l}])\cong{\mathcal U}_q(\mathfrak{so}_{10})$. We use the notation ${\bf I}^\prime=\{2,...,6\}$ as an index set for the simple roots of $\mathfrak{so}_{10}$. The algebra ${\mathcal U}_q^+[w]$ is a $\mathbb{Z}\{\varpi_i\}_{i\in{\bf I}^\prime}$-graded module, where the $\varpi_i$'s denote the fundamental weights. We refer to a nonzero element $x\in{\mathcal U}_q^+[w]$ as a highest weight vector of weight $\lambda\in\mathbb{Z}\{\varpi_i\}_{i\in{\bf I}^\prime}$ if $x$ is homogeneous of degree $\lambda$ and is annihilated by the positive root vectors of ${\mathcal U}_q(\mathfrak{l},\mathfrak{l}])$. For instance, the highest root vector $Y_\emptyset$ is a highest weight vector of weight $\varpi_2$, and
\[ \Theta:=Y_{[1234]}Y_{\emptyset}-qY_{[34]}Y_{[12]}+q^2Y_{[24]}Y_{[13]}-q^3Y_{[23]}Y_{[14]}\in{\mathcal U}_q^+[w]
\]
is a highest weight vector of weight $\varpi_6$. In Section \ref{adjoint action section}, we prove the following.

\begin{theorem}
As a ${\mathcal U}_q(\mathfrak{so}_{10})$-module, ${\mathcal U}_q^+[w]$ decomposes into a direct sum of submodules as follows:
\begin{equation}
{\mathcal U}_q^+[w] = \bigoplus_{m,n\geq 0} {\mathcal U}_q(\mathfrak{so}_{10}).(Y_{\emptyset}^m\Theta^n).
\end{equation}
Moreover, for every $m,n\geq 0$, ${\mathcal U}_q(\mathfrak{so}_{10}).(Y_{\emptyset}^m\Theta^n)$ is a finite-dimensional irreducible ${\mathcal U}_q(\mathfrak{so}_{10})$-module of highest weight $m\varpi_2+n\varpi_6$.
\end{theorem}

We give a conjecture for the decomposition of the quantum affine Schubert cell ${\mathcal U}_q^+[\widehat{w}]$ as a left ${\mathcal U}_q(\mathfrak{so}_{10})$-module.

\begin{conjecture}
The  vectors $\Omega_1,...,\Omega_{13}$ listed in Table \ref{HWVs table} are highest weight vectors for the adjoint action of ${\mathcal U}_q(\mathfrak{so}_{10})$ on ${\mathcal U}_q^+[\widehat{w}]$. As a ${\mathcal U}_q(\mathfrak{so}_{10})$-module, ${\mathcal U}_q^+[\widehat{w}]$ decomposes into a direct sum of finite-dimensional irreducible submodules as follows:
\[ \mathcal{U}_q^+[\widehat{w}] = \bigoplus_{{\bf r}} \mathcal{U}_q(\mathfrak{so}_{10}).(\Omega_1^{r_1}\cdots\Omega_{13}^{r_{13}}),\]
where the sum runs over all $r_1,...,r_{13}\in\mathbb{Z}_{\geq 0}$ such that $r_5r_9=0$.
\end{conjecture}

The universal FRT-bialgebra ${\mathcal A}$ of interest in this paper is the one obtained from the braiding on the quantum half-spin representation of ${\mathcal U}_q([\mathfrak{l},\mathfrak{l}])\cong{\mathcal U}_q(\mathfrak{so}_{10})$. The bialgebra ${\mathcal A}$ has a standard presentation in terms of a generating set $\{X_{IJ} : I,J\in{\mathcal B}\}$.  For $S\in{\mathcal B}$, let ${\mathcal A}_S$ be the subalgebra of ${\mathcal A}$ generated by $X_{SJ}$ ($J\in{\mathcal B}$).  Our first main result is that each ``row" of ${\mathcal A}$ is isomorphic to a quotient of ${\mathcal U}_q^+[w]$.

\begin{theorem}
For every $S\in{\mathcal B}$, there is a surjective algebra homomorphism  $\psi_S:  {\mathcal U}_q^+[w]\to {\mathcal A}_S$ given by
\begin{align*}
 Y_I  &\longmapsto X_{SI} .
\end{align*}
The kernel of $\psi_S$ is the ideal generated by the $10$-dimensional vector space ${\mathcal U}_q(\mathfrak{so}_{10}).\Theta$.
\end{theorem}

Our second main result involves the quantum affine Schubert cell algebra ${\mathcal U}_q^+[\widehat{w}]$. We label the quantum root vectors of  ${\mathcal U}_q^+[\widehat{w}]$ by $Z_I$ and $Z_{I+\delta}$ (for $I\in{\mathcal B}$). Following Artin, Schelter, and Tate \cite{AST91}, we twist the multiplication in ${\mathcal U}_q^+[\widehat{w}]$ by a $2$-cocycle to obtain a new algebra $({\mathcal U}_q^+[\widehat{w}])^\prime$, which has a slightly different multiplicative structure than ${\mathcal U}_q^+[\widehat{w}]$. We give a description of certain pairs of rows of generators of ${\mathcal A}$ in terms of the defining relations of $({\mathcal U}_q^+[\widehat{w}])^\prime$. For $S,T\in{\mathcal B}$ with $|(S\cup T)\backslash (S\cap T)|=2$ and $S<T$ (${\mathcal B}$ inherits a partial ordering from the root lattice),  let ${\mathcal A}_{(S,T)}$ denote the subalgebra of ${\mathcal A}$ generated by $X_{SJ}$, $X_{TJ}$ ($J\in{\mathcal B}$). We prove the following.

\begin{theorem}
 There is a surjective algebra homomorphism  $\psi_{(S,T)}:({\mathcal U}_q^+[\widehat{w}])^\prime \to {\mathcal A}_{(S,T)}$ given by
\begin{align*}
&(Z_I)^\prime \longmapsto X_{SI}, \\
&(Z_{I+\delta})^\prime\longmapsto X_{TI}.
\end{align*}
The kernel of $\psi_{(S,T)}$ is the ideal  generated by three $10$-dimensional vector spaces: $\left({\mathcal U}_q(\mathfrak{so}_{10}).\Omega_j\right)^\prime$ for $j=3,4,5$.
\end{theorem}

We remark that $\Omega_3$, $\Omega_4$, and $\Omega_5$ each have weight $\varpi_6$ (see Table \ref{HWVs table}). Hence, ${\mathcal U}_q(\mathfrak{so}_{10}).\Omega_3$, ${\mathcal U}_q(\mathfrak{so}_{10}).\Omega_4$, and ${\mathcal U}_q(\mathfrak{so}_{10}).\Omega_5$ are each isomorphic to the $10$-dimensional vector representation of ${\mathcal U}_q(\mathfrak{so}_{10})$. It is interesting to note that highest weight vectors of weight $0$ played an analogous role in the $D_n^{(1)}$ case.

\section{Preliminaries}

Let $\mathfrak{g}$ be a complex simple Lie algebra with Cartan subalgebra $\mathfrak{h}$, and let $\Phi\subseteq\mathfrak{h}^*$ be the root system of $\mathfrak{g}$. We denote the set of positive roots by $\Phi_+$, the negative roots by $\Phi_-$, and the set of simple roots by $\{\alpha_i\}_{i\in{\bf I}}$, where ${\bf I}$ is an index set. Let $\theta\in\Phi_+$ denote the highest root. Let $\{\varpi_i:i\in{\bf I}\}$ denote the set of fundamental weights of $\mathfrak{g}$. Let $Q=\mathbb{Z}\Phi$, $P=\mathbb{Z}\{\varpi_i:i\in{\bf I}\}$, and $Q_+=\mathbb{Z}_{\geq 0}\Phi$. Recall the partial ordering on $P$: $\mu\leq\mu^\prime$ if and only if $\mu^\prime-\mu\in Q_+$.  For $i\in{\bf I}$, let $s_i$ denote the simple reflection corresponding to $\alpha_i$, and let $W=\langle s_i:i\in{\bf I}\rangle\subseteq\text{Aut}_{\mathbb{Z}}(Q)$ denote the Weyl group of $\mathfrak{g}$. Let $\langle\,\,,\,\,\rangle$ be the invariant symmetric bilinear form on $\mathbb{R}\Phi$ such that $\langle\alpha,\alpha\rangle=2$ for short roots $\alpha\in\Phi$.

$ $\\
Throughout this paper, $k$ will denote a base field of arbitrary characteristic, and $q\in k^\times$ is not a root of unity.
$ $\\

Denote by ${\mathcal U}_q(\mathfrak{g})$ the quantized universal enveloping algebra over $k$ with deformation parameter $q$. The algebra ${\mathcal U}_q(\mathfrak{g})$ has generators $E_i$, $F_i$, and $K_i^{\pm1}$ ($i\in{\bf I}$) and defining relations given in \cite[\textsection 9.1]{CP}. Furthermore, ${\mathcal U}_q(\mathfrak{g})$  has a Hopf algebra structure with comultiplication $\Delta$, antipode $S$, and counit $\epsilon$ given by
\begin{align}
&\label{comultiplication formula}\Delta (E_i) = K_i^{-1}\otimes E_i+E_i\otimes 1,&          &\Delta (K_i ) = K_i\otimes K_i, &             &\Delta(F_i)= 1\otimes F_i+F_i\otimes K_i,\\
&S (E_i) = -K_iE_i,&                                                                  &S(K_i )=K_i^{-1},&                                               &S (F_i) = -F_iK_i^{-1},\\
&\epsilon (E_i) = 0,&                                                                                  &\epsilon (K_i ) =1,&                                              &\epsilon (F_i)=0.\end{align}

Let $B_{\mathfrak{g}}$ be the braid group of $\mathfrak{g}$ and let $\{T_{s_i}:i\in{\bf I}\}$ denote its standard generating set.  There is a standard action by algebra automorphisms of $B_{\mathfrak{g}}$ on $U_q(\mathfrak{g})$ due to Lusztig; we use the version of Lusztig's action given by the formulas in \cite[\textsection 2.3]{Y}.
To each reduced expression of the longest element $w_0\in W$, the braid group action is used to construct a PBW basis of ${\mathcal U}_q(\mathfrak{g})$, see, e.g., \cite[I.6.7]{BG}, \cite[\textsection 9.1.B]{CP}, or \cite[Ch. 37]{L}.

Fix $w\in W$. We recall how the braid group action is used to construct a PBW basis for the quantum Schubert cell algebra ${\mathcal U}_q^+[w]$. For a reduced expression \begin{equation}w = s_{i_1}\cdots s_{i_t}\end{equation} define the roots
\begin{equation}
\beta_1 = \alpha_{i_1}, \beta_2=s_{i_1}\alpha_{i_2},...,\beta_t = s_{i_1}\cdots s_{i_{t-1}}\alpha_{i_t},\end{equation}
and positive root vectors
\begin{equation}
X_{\beta_1}=E_{i_1}, X_{\beta_2}=T_{s_{i_1}}E_{i_2},...,X_{\beta_t}=T_{s_{i_1}}\cdots T_{s_{i_{t-1}}}E_{i_t}.
\end{equation}
Let $\Phi_w$ denote the set of radical roots $\{\beta_1,...,\beta_t\}$. The radical roots of $w$ are precisely the positive roots that get sent to negative roots by the action of $w^{-1}$.
Following \cite{DKP}, let ${\mathcal U}_q^+[w]$ denote the subalgebra of ${\mathcal U}_q(\mathfrak{g})$ generated by the positive root vectors $X_{\beta_1},...,X_{\beta_t}$. The algebra ${\mathcal U}_q(\mathfrak{g})$ is $Q$-graded by setting $\text{deg}(K_i)=0$, $\text{deg}(E_i)=\alpha_i$, and $\text{deg}(F_i)=-\alpha_i$. This induces a $Q$-grading on ${\mathcal U}_q^+[w]$. De Concini, Kac, and Procesi \cite[Proposition 2.2]{DKP} proved that the algebra ${\mathcal U}_q^+[w]$ does not depend on the reduced expression for $w$. Furthermore, ${\mathcal U}_q^+[w]$ has the PBW basis
\begin{equation}
X_{\beta_1}^{n_1}\cdots X_{\beta_t}^{n_t},\hspace{.4cm} n_1,...,n_t\in\mathbb{Z}_{\geq 0}.
\end{equation}

Beck later proved the analogous result for the case when $\mathfrak{g}$ is an affine Kac-Moody Lie algebra \cite{B1,B2}.  For $char(k)=0$, the Levendorskii-Soibelmann Straightening Rule gives commutation relations in ${\mathcal U}_q^+[w]$.

\begin{theorem}\cite[Prop. 5.5.2]{LS} For $i<j$,
\begin{equation}
X_{\beta_i}X_{\beta_j} = q^{\langle\beta_i,\beta_j\rangle}X_{\beta_j}X_{\beta_i}+\sum_{n_{i+1},...,n_{j-1}\geq 0} z(n_{i+1},...,n_{j-1})X_{\beta_{i+1}}^{n_{i+1}}\cdots X_{\beta_{j-1}}^{n_{j-1}},
\end{equation}
where $z(n_{i+1},...,n_{j-1})\in\mathbb{Q}[q^{\pm 1}]$
\end{theorem}

It follows that ${\mathcal U}_q^+[w]$ is an iterated Ore extension over the base field $k$. We remark that a straightening rule holds for the algebras of interest in this paper, even though we do not assume $char(k)=0$.

\section{Root Data for Types $D_5$, \texorpdfstring{$E_6$}{E_6}, and \texorpdfstring{$E_6^{(1)}$}{E_6^(1)}}\label{section root system E6}

In this section we will give an explicit realization of the root systems in which we are interested. Let
\[  {\bf I^\prime}=\{2,...,6\}, \hspace{.6in} {\bf I}=\{1,...,6\}, \hspace{.6in}{\bf I_0}=\{0,1,...,6\}.\]
These sets are index sets for the Lie algebras of type $D_5$, $E_6$, and $E_6^{(1)}$ respectively. For ${\bf J}={\bf I^\prime},{\bf I}$ or ${\bf I_0}$, let ${\mathcal U}_q(\mathfrak{g}_{{\bf J}})$ denote the corresponding quantized enveloping algebra with generators $E_i$, $F_i$, $K_i^{\pm 1}$ ($i\in{\bf J}$) and defining relations given in \cite[\textsection 9.1]{CP}.  Since  ${\bf I^\prime}\subseteq{\bf I}\subseteq{\bf I_0}$, this yields a natural chain of inclusions ${\mathcal U}_q(\mathfrak{g}_{{\bf I^\prime}})\subseteq{\mathcal U}_q(\mathfrak{g}_{{\bf I}})\subseteq{\mathcal U}_q(\mathfrak{g}_{{\bf I_0}})$.

$ $\\

Throughout this paper we treat all subalgebras and sub-Hopf algebras of ${\mathcal U}_q(\mathfrak{g}_{{\bf I^\prime}})$ and ${\mathcal U}_q(\mathfrak{g}_{{\bf I}})$ as subalgebras and sub-Hopf algebras of ${\mathcal U}_q(\mathfrak{g}_{{\bf I_0}})$.

$ $\\

Let $e_1,...,e_6$ denote an orthonormal basis of a six-dimensional Euclidean space $E$ with respect to an inner product $\left<\,\,\,,\,\,\,\right>$, and define a subset $\Phi\subset E$ as follows:
\[ \Phi := \left\{ x=(a_1,a_2,a_3,a_4,a_5,\sqrt{3}a_6)\in E: (a_1,...,a_5,a_6)\in\mathbb{Z}^6\cup\left(\mathbb{Z}+\frac{1}{2}\right)^6,a_6=-16a_1a_2a_3a_4a_5, |x|^2=2\right\}. \]
The set $\Phi$ is a root system of type $E_6$. The roots $\alpha_1=\frac{1}{2}(e_1-e_2-e_3-e_4-e_5-\sqrt{3}e_6)$, $\alpha_2=e_1+e_2$, and $\alpha_i= e_{i-1}-e_{i-2}$ for $3\leq i\leq 6$ are a basis of positive simple roots for $\Phi$. This choice of positive roots yields a decomposition $\Phi =\Phi_+ \cup \Phi_-$ into positive and negative roots.  Let $Q=\mathbb{Z}\Phi$, $Q_+=\mathbb{Z}_{\geq 0}\Phi_+$, and $P=\mathbb{Z}\{\varpi_i:i\in{\bf I}\}$.  Recall the partial ordering on $P$ is defined by $\beta\leq\beta^\prime$ if and only if $\beta^\prime-\beta\in Q_+$. The partial ordering on $P$ induces on ordering on any subset of $P$. In particular, this gives a partial ordering on $\Phi$ and $Q$.

A basis for the root system of type $E_6^{(1)}$ is obtained from the basis $\{\alpha_i\}_{i\in{\bf I}}$ of $\Phi$ by including an additional root $\alpha_0:=-\theta+\delta$, where $\theta=\alpha_1+2\alpha_2+2\alpha_3+3\alpha_4+2\alpha_5+\alpha_6\in\Phi$ is the highest root and $\delta$ is isotropic. The numbering of the simple roots is compatible with the numbering given by the Dynkin diagram in Figure \ref{Dynkin diagram}.

\begin{figure}[t]\caption{Dynkin Diagram for $E_6^{(1)}$}\label{Dynkin diagram}
\begin{picture}(200,110)\linethickness{2pt}
	\put(3,14){\circle*{5}}
	\put(53,14){\circle*{5}}
	\put(103,14){\circle*{5}}
	\put(153,14){\circle*{5}}
	\put(203,14){\circle*{5}}
	\put(103,57){\circle*{5}}
	\put(103,100){\circle*{5}}
	\put(5,14){\line(1,0){200}}
	\put(103,14){\line(0,1){86}}
	\put(107,97){0}
	\put(0,0){1}
	\put(107,54){2}
	\put(50,0){3}
	\put(100,0){4}
	\put(150,0){5}
	\put(200,0){6}
\end{picture}
\end{figure}
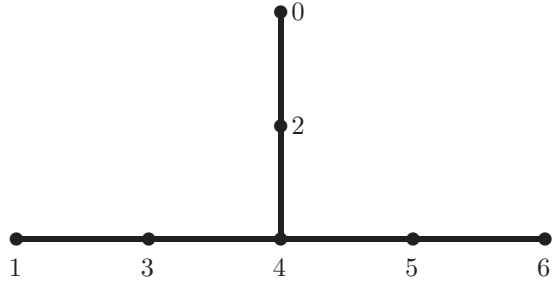

Denote by $W_{{\bf I_0}}=\langle s_i:i\in{\bf I_0}\rangle$ the affine Weyl group of type $E_6^{(1)}$. For any subset ${\bf J}\subseteq{\bf I_0}$, let $W_{{\bf J}}$ be the subgroup of $W_{{\bf I_0}}$ generated by $\{s_j: j\in{\bf J}\}$. Thus $W_{{\bf I}}$ is the finite Weyl group of type $E_6$ and $W_{{\bf I^\prime}}$ is the Weyl group of type $D_5$. Let $w_0^{\bf J}$ denote the longest element of $W_{\bf J}$ (if it exists).

\section{The Quantum Schubert Cell Algebras ${\mathcal U}_q^+[w]$ and \texorpdfstring{${\mathcal U}_q^+[\widehat{w}]$}{U_q^+[hat w]}}\label{section quantum schubert cells}

Let
\begin{equation}
w=w_0^{{\bf I^\prime}}w_0^{\bf I}\in W_{{\bf I}}\subseteq W_{{\bf I_0}}.
\end{equation}
We have the following reduced expression
\begin{equation}
w=(s_1s_3s_4s_2s_5s_4s_3s_1)(s_6s_5s_4s_3s_2s_4s_5s_6)\in W_{{\bf I}},
\end{equation}
and radical roots $\Phi_w:=\{\beta\in\Phi:\beta\geq\alpha_1\}$.  We associate to $w$ a related Weyl group element $\widehat{w}$ as follows. Let  $\widehat{w}\in W_{{\bf I_0}}$ be the Weyl group element so that the set of  radical roots of $\widehat{w}$ is
\[ \Phi_{\widehat{w}}:=\{\beta+n\delta:\beta\in\Phi_w,n=0,1\}.\]
The length of $\widehat{w}$ is twice the length of $w$. We can write $\widehat{w}$ as follows:
\begin{equation}
 \widehat{w} = w(s_0s_2s_4s_5s_3s_4s_2s_0)(s_1s_3s_4s_2s_5s_4s_3s_1)\in W_{{\bf I_0}}.
 \end{equation}
More generally, if $u$ is an arbitrary element of a \emph{finite} Weyl group, then there is a $\widehat{u}$ in the corresponding affine Weyl group so that $\Phi_{\widehat{u}}=\{\beta+n\delta:\beta\in\Phi_u,n=0,1\}$ if and only if $\Phi_u=\{\beta : \beta\geq \alpha_i\}$ for some cominiscule root $\alpha_i$. Furthermore, the element $\widehat{u}$ constructed from $u$ is unique whenever it exists.

Denote by ${\mathcal U}_q^+[w]$ and ${\mathcal U}_q^+[\widehat{w}]$ the quantum Schubert cell algebras corresponding to $w$ and $\widehat{w}$. Observe first that the highest root $\theta\in\Phi$ is equal to $\varpi_2$ and
\[\Phi_w=\{\varpi_2 - \sum_{i\in I}e_i:I\in{\mathcal B}\},
\]
where ${\mathcal B}$ is the set of all subsets of $\{1,2,3,4,5\}$ having even cardinality.
Hence, we will identify a root $\beta\in\Phi_w$ (or a root vector $X_\beta\in{\mathcal U}_q^+[w]$) with the corresponding set $I\in{\mathcal B}$ whenever $\beta=\varpi_2-\sum_{i\in I}e_i\in\Phi_w$.  The poset structure on $\Phi_w$ induces a poset structure on  ${\mathcal B}$ (see Figure \ref{poset B}).  Define the weight of a subset $I\in{\mathcal B}$ by $wt(I)=\varpi_2-\sum_{i\in I}e_i\in\Phi_w$.

We define an equivalence relation on ${\mathcal B}\times{\mathcal B}$ by the rule $(I,J)\sim (K,L)$ if and only if $wt(I)+wt(J)=wt(K)+wt(L)$. There are $16$ equivalence classes of size $1$. They are each of the form $\{(I,I)\}$ for $I\in{\mathcal B}$. Furthermore, for every $I,J\in{\mathcal B}$ such that $|(I\cup J)\backslash (I\cap J)|=2$, $\{(I,J),(J,I)\}$ is an equivalence class. There are $80$ such equivalence classes of this form. Finally there are $10$ classes of size $8$. In fact $(I,J)$ is in an equivalence class of size $8$ if and only if $|(I\cup J)\backslash (I\cap J)|=4$ (see Table \ref{large classes}).  Recalling that the weight function on $\mathcal B$ is injective, we define a partial ordering $\preceq$ on ${\mathcal B}\times{\mathcal B}$ by $(I,J)\preceq (K,L)$ if and only if $(I,J)\sim (K,L)$ and $I\leq K$. The partial ordering makes ${\mathcal B}\times{\mathcal B}$ a graded poset. We will let $ht:{\mathcal B}\times{\mathcal B}\to\mathbb{N}$ denote the height function on this poset such that the minimal elements have height $1$.

The following proposition gives another description of the equivalence relation on ${\mathcal B}\times{\mathcal B}$ and a straightforward method to determine the size of an equivalence class of $(I,J)\in{\mathcal B}\times{\mathcal B}$ based on the size of the symmetric difference $(I\cup J)\backslash (I\cap J)$.

\begin{proposition}
For all $I,J,K,L\in{\mathcal B}$, we have the following:
\begin{enumerate}[(i)]
\item $(I,J)\sim (K,L)$ if and only if $I\cup J=K\cup L$ and $I\cap J=K\cap L$.
\item $\left<wt(I),wt(J)\right>=2-\frac{1}{2}|(I\cup J)\backslash (I\cap J)|$, which is equal to $0$, $1$, or $2$ whenever the cardinality of the equivalence class $[(I,J)]$ is $8$, $2$, or $1$ respectively.
\end{enumerate}
\end{proposition}

We label the root vectors of ${\mathcal U}_q^+[w]$ by $Y_I$ (for $I\in{\mathcal B}$) and the root vectors of ${\mathcal U}_q^+[\widehat{w}]$ by $Z_I$ and $Z_{I+\delta}$ (for $I\in{\mathcal B}$). In order to succinctly write the defining relations of ${\mathcal U}_q^+[w]$ and ${\mathcal U}_q^+[\widehat{w}]$, we first need to introduce some notation. We define a function $N:{\mathcal B}\to\mathbb{N}\cup\{0\}$ as follows: if $I=\{i_1>\cdots >i_\ell\}\in{\mathcal B}$, then $N(I)=\sum_{1\leq j\leq \ell} 10^{j-1}i_j$. We use the convention that $N(\emptyset)=0$. Define the lexicographic ordering $\leq_{lexico.}$ on ${\mathcal B}$ by the rule $I\leq_{lexico.}J$ if and only if $N(I)\leq N(J)$. We define the following for every $I,J\in{\mathcal B}$:
\begin{align}
\label{epsilondefn}&\epsilon(I,J):=\begin{cases}1, &\text{if }I<J\text{ and }\langle wt(I),wt(J)\rangle=0,\\
0, &\text{otherwise.}\end{cases}
\end{align}

The following theorems give the defining relations of ${\mathcal U}_q^+[w]$ and ${\mathcal U}_q^+[\widehat{w}]$.

\begin{theorem} The algebra ${\mathcal U}_q^+[w]$ is generated by $\{Y_I:I\in{\mathcal B}\}$ and has defining relations
\begin{equation}
Y_IY_J =q^{\langle wt(I),wt(J)\rangle}Y_JY_I+\hat{q}\sum_{\tiny\begin{array}{c}(I,J)\prec (L,M),\\ M\leq_{lexico.} L\end{array}}  (-q)^{ht(L,M)-ht(I,J)-1}Y_LY_M,
\end{equation}
for every $I,J\in{\mathcal B}$ such that $I\not \geq J$.
\end{theorem}

\begin{theorem} The algebra ${\mathcal U}_q^+[\widehat{w}]$ is generated by $\{Z_I,Z_{I+\delta}:I\in{\mathcal B}\}$ and has defining relations
 \begin{align}
\label{Uqw_hat relns1}&Z_{I+n\delta}Z_{J+n\delta} =q^{\langle wt(I),wt(J)\rangle} Z_{J+n\delta}Z_{I+n\delta}+\hat{q}\!\!\!\!\!\!\!\sum_{\tiny\begin{array}{c}(I,J)\prec (L,M),\\ M\leq_{lexico.} L\end{array}}\!\!\!\!\!\!\!  (-q)^{ht(L,M)-ht(I,J)-1}Z_{L+n\delta}Z_{M+n\delta},\\
&Z_IZ_{J+\delta} =q^{\langle wt(I),wt(J)\rangle} Z_{J+\delta} Z_I+\hat{q}\left(\sum_{(I,J)\prec(L,M)}\!\!\!\!\!\!\!(-q)^{ht(L,M)-ht(I,J)-1}Z_LZ_{M+\delta}+q^{-1}\epsilon(I,J)Z_JZ_{I+\delta}\right),
\end{align}
for every $I,J\in{\mathcal B}$ and $n\in\{0,1\}$ (with $I\not\geq J$ in Eqn. \ref{Uqw_hat relns1} above).
\end{theorem}

\begin{table}[t]\caption{All pairs of sets $(I,J)\in{\mathcal B}\times{\mathcal B}$ such that the size of the equivalence class $[(I,J)]$ is $8$. Each row is an equivalence class.}\label{large classes}
\begin{tabular}{|c c c c c c c c|}
\hline
{\bf Height=1} & {\bf Height=2} & {\bf Height=3} & {\bf Height=4} &{\bf Height=4} &{\bf Height=5} &{\bf Height=6} &{\bf Height=7}\\
\hline \hline
 (1234,$\emptyset$)    & (34,12)    & (24,13)    & (23,14)    & (14,23)    & (13,24)     & (12,34)    & ($\emptyset$,1234) \\ \hline
 (1235,$\emptyset$)    & (35,12)    & (25,13)    & (23,15)    & (15,23)    & (13,25)     & (12,35)    & ($\emptyset$,1235) \\ \hline
 (1245,$\emptyset$)    & (45,12)    & (25,14)    & (24,15)    & (15,24)    & (14,25)     & (12,45)    & ($\emptyset$,1245) \\ \hline
 (1345,$\emptyset$)    & (45,13)    & (35,14)    & (34,15)    & (15,34)    & (14,35)     & (13,45)    & ($\emptyset$,1345) \\ \hline
 (2345,$\emptyset$)    & (45,23)    & (35,24)    & (34,25)    & (25,34)    & (24,35)     & (23,45)    & ($\emptyset$,2345) \\ \hline
 (1345,12)    & (1245,13)    & (1235,14)    & (1234,15)    & (15,1234)    & (14,1235)     & (13,1245)    & (12,1345) \\ \hline
 (2345,12)    & (1245,23)    & (1235,24)    & (1234,25)    & (25,1234)    & (24,1235)     & (23,1245)    & (12,2345) \\ \hline
 (2345,13)    & (1345,23)    & (1235,34)    & (1234,35)    & (35,1234)    & (34,1235)     & (23,1345)    & (13,2345) \\ \hline
 (2345,14)    & (1345,24)    & (1245,34)    & (1234,45)    & (45,1234)    & (34,1245)     & (24,1345)    & (14,2345) \\ \hline
 (2345,15)    & (1345,25)    & (1245,35)    & (1235,45)    & (45,1235)    & (35,1245)     & (25,1345)    & (15,2345) \\ \hline
\end{tabular}\end{table}

\section{The Adjoint Action of \texorpdfstring{${\mathcal U}_q(\mathfrak{g}_{{\bf I^\prime}})$}{Uq(g_{I^\prime})} on ${\mathcal U}_q^+[w]$ and \texorpdfstring{${\mathcal U}_q^+[\widehat{w}]$}{U_q^+[hat w]}}\label{adjoint action section}

In this section we identify certain highest weight vectors for the adjoint action of ${\mathcal U}_q(\mathfrak{g}_{{\bf I^\prime}})$ on the quantum Schubert cell algebras ${\mathcal U}_q^+[w]$ and ${\mathcal U}_q^+[\widehat{w}]$. We will later see that these highest weight vectors will play  an important role in the main results of Section \ref{section_mainresults}. We begin with the following theorem.
\begin{theorem}\label{action is defined}
The algebras ${\mathcal U}_q^+[w]$ and ${\mathcal U}_q^+[\widehat{w}]$ are stable under the adjoint action of ${\mathcal U}_q(\mathfrak{g}_{{\bf I^\prime}})$. Moreover, ${\mathcal U}_q^+[w]$ and ${\mathcal U}_q^+[\widehat{w}]$ are left ${\mathcal U}_q(\mathfrak{g}_{{\bf I^\prime}})$-module algebras with respect to this action.
\end{theorem}
\begin{proof} For $\beta\in\Phi_w$ or $\Phi_{\widehat{w}}$, let $X_\beta$ be the root vector (in ${\mathcal U}_q^+[w]$ or ${\mathcal U}_q^+[\widehat{w}]$) of degree $\beta$. One can verify that
\begin{align}\label{action on X_betas}
ad(E_j) X_\beta =\begin{cases} -qX_{\beta+\alpha_j}, &\text{if }\beta+\alpha_j\in\Phi_w (\text{or }\Phi_{\widehat{w}}),\\ 0, & otherwise, \end{cases}& &ad(F_j) X_\beta =\begin{cases} -q^{-1}X_{\beta-\alpha_j}, &\text{if }\beta-\alpha_j\in\Phi_w (\text{or }\Phi_{\widehat{w}}),\\ 0, & otherwise, \end{cases}
 \end{align}
for all $j\in {\bf I^\prime}$. Since ${\mathcal U}_q(\mathfrak{g}_{{\bf I_0}})$ is a left ${\mathcal U}_q(\mathfrak{g}_{{\bf I_0}})$-module algebra with respect to the adjoint action, then the equations above, together with the fact that the $K_i$'s act diagonally, prove the desired result.
\end{proof}

The adjoint action induces a $\mathbb{Z}\{\varpi_i\}_{i\in{\bf I^\prime}}$-grading on the algebras ${\mathcal U}_q^+[w]$ and ${\mathcal U}_q^+[\widehat{w}]$. For $\lambda\in\mathbb{Z}\{\varpi_i\}_{i\in{\bf I^\prime}}$, the homogeneous component of degree $\lambda$ is given by
\begin{align*}
&{\mathcal U}_q^+[w] _\lambda = \{x\in {\mathcal U}_q^+[w] : ad(K_i).x=q^{\left<\alpha_i,\lambda\right>}x\text{ for all } i\in{\bf I^\prime}\},\\
&{\mathcal U}_q^+[\widehat{w}] _\lambda = \{x\in {\mathcal U}_q^+[\widehat{w}] : ad(K_i).x=q^{\left<\alpha_i,\lambda\right>}x\text{ for all } i\in{\bf I^\prime}\}.
\end{align*}

We are interested in obtaining explicit decompositions of ${\mathcal U}_q^+[w]$ and ${\mathcal U}_q^+[\widehat{w}]$ into irreducible submodules because certain highest weight vectors play an important role in the main results of Section \ref{section_mainresults}. Define the subalgebras $A[w]\subseteq{\mathcal U}_q^+[w]$ and $A[\widehat{w}]\subseteq{\mathcal U}_q^+[\widehat{w}]$ as follows:
\begin{align*}
&A[w]:=\{x\in{\mathcal U}_q^+[w]:ad(E_i).x=0\text{ for all }i\in{\bf I^\prime}\},\\
&A[\widehat{w}]:=\{x\in{\mathcal U}_q^+[\widehat{w}]:ad(E_i).x=0\text{ for all }i\in{\bf I^\prime}\}.
\end{align*}

\begin{table}[t]
\caption{A list of conjectured generators for the algebra $A[\widehat{w}]$.}
\label{HWVs table}
\begin{tabular}{|l c c|}
\hline
  {\bf Highest Weight Vector} & \begin{tabular}{c}{\bf degree w.r.t.}\\ $\mathbb{Z}\{\varpi_i\}_{i\in{\bf I}^\prime}${\bf -grading}\end{tabular} &\begin{tabular}{c} {\bf degree w.r.t.}\\$\mathbb{Z}_{\geq 0}${\bf -grading}\end{tabular} \\
\hline \hline
$\Omega_1:=Z_{\emptyset}$ & $\varpi_2$    & 1  \\ \hline
$\Omega_2:=Z_{\emptyset+\delta}$ & $\varpi_2$    & 1  \\ \hline
$\Omega_3:=Z_{[1234]}Z_{\emptyset}-qZ_{[34]}Z_{[12]}+q^2Z_{[24]}Z_{[13]}-q^3Z_{[23]}Z_{[14]}$ & $\varpi_6$    & 2  \\ \hline
$\Omega_4:=\sum_{(I,J)\sim (1234,\emptyset)} (-q)^{ht(I,J)}Z_IZ_{J+\delta}$& $\varpi_6$ & 2  \\ \hline
$\!\!\!\begin{array}{ll}\Omega_5:=&\!\!\!Z_{[1234]+\delta}Z_{\emptyset+\delta}-qZ_{[34]+\delta}Z_{[12]+\delta}\\&+q^2Z_{[24]+\delta}Z_{[13]+\delta}-q^3Z_{[23]+\delta}Z_{[14]+\delta}\end{array}$& $\varpi_6$ & 2  \\ \hline
$\Omega_6:=(ad(F_2)\Omega_1)\cdot\Omega_2-q\Omega_1\cdot(ad(F_2)\Omega_2)$ & $\varpi_4$ & 2 \\ \hline
$\!\!\!\begin{array}{ll}\Omega_7:=&\!\!\!(ad(F_2F_4F_5F_6)\Omega_3)\cdot\Omega_2\\&-q(ad(F_4F_5F_6)\Omega_3)\cdot(ad(F_2)\Omega_2)\\&+q^2(ad(F_5F_6)\Omega_3)\cdot(ad(F_4F_2)\Omega_2)\\&-q^3(ad(F_6)\Omega_3)\cdot(ad(F_5F_4F_2)\Omega_2)\\&+q^4\Omega_3\cdot(ad(F_6F_5F_4F_2)\Omega_2)\end{array}$ & $\varpi_3$ & 3 \\ \hline
$\!\!\!\begin{array}{ll}\Omega_8:=&\!\!\!\Omega_1\cdot(ad(F_2F_4F_5F_6)\Omega_5)\\
&-q^{-1}(ad(F_2)\Omega_1)\cdot(ad(F_4F_5F_6)\Omega_5)\\
&+q^{-2}(ad(F_4F_2)\Omega_1)\cdot(ad(F_5F_6)\Omega_5)\\
&-q^{-3}(ad(F_5F_4F_2)\Omega_1)\cdot(ad(F_6)\Omega_5)\\
&+q^{-4}(ad(F_6F_5F_4F_2)\Omega_1)\cdot\Omega_5
\end{array}$ &$\varpi_3$ & 3 \\ \hline
$\Omega_9:=\Omega_3\cdot(ad(F_6)\Omega_4)-q^{-1}(ad(F_6)\cdot\Omega_3)\Omega_4$ & $\varpi_5$ & 4 \\ \hline
$\Omega_{10}:=\Omega_3\cdot(ad(F_6)\Omega_5)-q^{-1}(ad(F_6)\Omega_3)\cdot\Omega_5$ & $\varpi_5$ & 4 \\ \hline
$\Omega_{11}:=\Omega_4\cdot(ad(F_6)\Omega_5)-q^{-1}(ad(F_6)\Omega_4)\cdot\Omega_5$ & $\varpi_5$ & 4 \\ \hline
$\!\!\!\begin{array}{ll}\Omega_{12}:=&\!\!\!(ad(F_6F_5F_4F_3F_2F_4F_5F_6)\Omega_3)\cdot\Omega_5\\
&-q(ad(F_5F_4F_3F_2F_4F_5F_6)\Omega_3)\cdot(ad(F_6)\Omega_5)\\
&+q^2(ad(F_4F_3F_2F_4F_5F_6)\Omega_3)\cdot(ad(F_5F_6)\Omega_5)\\
&-q^3(ad(F_3F_2F_4F_5F_6)\Omega_3)\cdot(ad(F_4F_5F_6)\Omega_5)\\
&+q^4(ad(F_2F_4F_5F_6)\Omega_3)\cdot(ad(F_3F_4F_5F_6)\Omega_5)\\
&+q^4(ad(F_3F_4F_5F_6)\Omega_3)\cdot(ad(F_2F_4F_5F_6)\Omega_5)\\
&-q^5(ad(F_4F_5F_6)\Omega_3)\cdot(ad(F_3F_2F_4F_5F_6)\Omega_5)\\
&+q^6(ad(F_5F_6)\Omega_3)\cdot(ad(F_4F_3F_2F_4F_5F_6)\Omega_5)\\
&-q^7(ad(F_6)\Omega_3)\cdot(ad(F_5F_4F_3F_2F_4F_5F_6)\Omega_5)\\
&+q^8\Omega_3\cdot(ad(F_6F_5F_4F_3F_2F_4F_5F_6)\Omega_5)\end{array}$ & $0$ & 4 \\ \hline
$\!\!\!\begin{array}{ll}\Omega_{13}:=&\!\!\!\Omega_3\cdot(ad(F_6F_5)\Omega_{11})\\
&-q^{-1}(ad(F_6)\Omega_3)\cdot (ad(F_55)\Omega_{11})\\
&+q^{-2}(ad(F_5F_6)\Omega_3)\cdot\Omega_{11}\end{array}$ & $\varpi_4$ & 6 \\ \hline
\end{tabular}\end{table}

We refer to a nonzero element $x\in A[w]\cap{\mathcal U}_q^+[w] _\lambda$ (or $x\in A[\widehat{w}]\cap{\mathcal U}_q^+[\widehat{w}] _\lambda$) as a highest weight vector of weight $\lambda$. We have the following.

\begin{proposition}\label{HWVs2}
The vector
\begin{equation}\Theta:=Y_{[1234]}Y_{\emptyset}-qY_{[34]}Y_{[12]}+q^2Y_{[24]}Y_{[13]}-q^3Y_{[23]}Y_{[14]}\in{\mathcal U}_q^+[w]
\end{equation}
is a highest weight vector of weight $\varpi_6$.
\end{proposition}

As a ${\mathcal U}_q(\mathfrak{g}_{{\bf I^\prime}})$-module, ${\mathcal U}_q^+[w]$ decomposes into a direct sum of finite-dimensional irreducible submodules. In Theorem \ref{Uqw decomposition} below, we give an explicit description of this decomposition, but first we need the following lemma.

\begin{lemma}\label{combinatorial identity}
For all $d\geq 0$,
\begin{equation}
\sum_{\tiny\begin{array}{c} m,n\geq 0\\ m+2n=d\end{array}} \frac{m+3}{105} {m+5 \choose 5}{n+4 \choose 4}{n+m+7 \choose 4} =
{15+d \choose 15}
\end{equation}
\end{lemma}
\begin{proof}
Viewing $d$ as a formal variable allows us to view each side of the above identity as a polynomial in $\mathbb{Q}[d]$. Observe that the right-hand side ${15 +d \choose 15}$ has degree $15$. The left-hand side also has degree $15$. This follows from the fact that $p(m,n):=\frac{m+3}{105}{m+5 \choose 5}{n+4 \choose 4}{n+m+7 \choose 4}$, as a polynomial in $m$ and $n$, has total degree $14$. Substituting $m=d-2n$ everywhere in the expression $p(m,n)$ yields a polynomial in $d$ and $n$ of total degree $14$. Finally, summing $n$ over the interval $0\leq n \leq d/2$ gives a polynomial in $d$ of degree $15$. Therefore it is sufficient to verify the identity holds for $16$ values of $d$.
\end{proof}

The following theorem highlights the role $\Theta$ plays in the adjoint action.

\begin{theorem}\label{Uqw decomposition}
As a ${\mathcal U}_q(\mathfrak{g}_{{\bf I^\prime}})$-module, ${\mathcal U}_q^+[w]$ decomposes into a direct sum of submodules as follows:
\begin{equation}
{\mathcal U}_q^+[w] = \bigoplus_{m,n\geq 0} \text{ad}({\mathcal U}_q(\mathfrak{g}_{{\bf I^\prime}})).(Y_{\emptyset}^m\Theta^n).
\end{equation}
Moreover, for every $m,n\geq 0$, $\text{ad}({\mathcal U}_q(\mathfrak{g}_{{\bf I^\prime}})).(Y_{\emptyset}^m\Theta^n)$ is a finite-dimensional irreducible ${\mathcal U}_q(\mathfrak{g}_{{\bf I^\prime}})$-module of highest weight $m\varpi_2+n\varpi_6$.
\end{theorem}

\begin{proof}
Since $Y_{\emptyset}$ and $\Theta$ are highest weight vectors, it follows that $Y_{\emptyset}^m\Theta^n$ is also a highest weight vector for every $m,n\geq 0$.  The vector $Y_{\emptyset}^m\Theta^n$ has weight $m\varpi_2+n\varpi_6$. The Weyl dimension formula implies
\begin{align*}
\text{dim}_k\left(\text{ad}({\mathcal U}_q(\mathfrak{g}_{{\bf I^\prime}}).(Y_{\emptyset}^m\Theta^n)\right) &= \frac{\prod_{\alpha\in(\Phi_{{\bf I^\prime}})_+}\langle m\varpi_2+n\varpi_6+\rho,\alpha\rangle}{\prod_{\alpha\in(\Phi_{{\bf I^\prime}})_+}\langle\rho,\alpha\rangle}\\
&=\frac{m+3}{105}{m+5 \choose 5}{n+4 \choose 4}{n+m+7 \choose 4},
\end{align*}
where $\rho = \frac{1}{2}\sum_{\alpha_\in(\Phi_{{\bf I^\prime}})_+}\alpha$. Observe that ${\mathcal U}_q^+[w]$ has a $\mathbb{Z}_{\geq 0}$-grading given by $\text{deg}(Y_I)=1$ for all $I\in{\mathcal B}$, and the adjoint action of ${\mathcal U}_q(\mathfrak{g}_{{\bf I^\prime}})$ preserves this grading. The dimension (over $k$) of the degree $d$ part of ${\mathcal U}_q^+[w]$ is ${15+d \choose 15}$, the same as an ordinary polynomial ring in $16$ commuting variables.  Thus, Lemma \ref{combinatorial identity} implies that counting the  dimension on each homogeneous component $({\mathcal U}_q^+[w])_d$ ($d=0,1,...$) gives the desired result.
\end{proof}

The highest weight vectors $Y_\emptyset\in{\mathcal U}_q^+[w]$ and $\Theta\in{\mathcal U}_q^+[w]$ commute. Therefore $A[w]$ is isomorphic to a polynomial ring in two commuting variables. We make the following conjecture concerning the algebra $A[\widehat{w}]$.

\begin{conjecture}\label{conj1}
$ $
\begin{enumerate}[(i)]
\item The  vectors $\Omega_1,...,\Omega_{13}$ listed in Table \ref{HWVs table} are highest weight vectors. The set of monomials $\{\Omega_1^{r_1}\cdots\Omega_{13}^{r_{13}}\}$ with $r_5r_9=0$ is a $k$-basis of $A[\widehat{w}]$.
\item As a ${\mathcal U}_q(\mathfrak{g}_{{\bf I^\prime}})$-module, ${\mathcal U}_q^+[\widehat{w}]$ decomposes into a direct sum of finite-dimensional irreducible submodules as follows:
\[ \mathcal{U}_q^+[\widehat{w}] = \bigoplus_{{\bf r}} ad(\mathcal{U}_q(\mathfrak{g}_{{\bf I^\prime}})).(\Omega_1^{r_1}\cdots\Omega_{13}^{r_{13}}),\]
where the sum runs over all $r_1,...,r_{13}\in\mathbb{Z}_{\geq 0}$ such that $r_5r_9=0$.
\end{enumerate}
\end{conjecture}

We remark that the condition $r_5r_9=0$ in Conjecture \ref{conj1} comes from a linear dependence relation among ordered monomials. We conjecture that there are $13\choose 2$ +1 defining relations for the algebra $A[\widehat{w}]$: $13\choose 2$ commutation relations implying ordered monomials span $A[\widehat{w}]$ together with the linear dependence relation $\Omega_5\Omega_9=-q^{-6}\Omega_3\Omega_{11}+q^{-2}\Omega_4\Omega_{10}$. Some evidence for the validity of part $(ii)$ of Conjecture \ref{conj1} can be obtained by a dimension counting argument similar to the proof of Theorem \ref{Uqw decomposition}.

The highest weight vectors $\Theta$, $\Omega_3$, $\Omega_4$, $\Omega_5$ will play an important role in the main results of Section \ref{section_mainresults}.  Before we present the main results it will be necessary to recall the details on universal $R$-matrices and the universal bialgebra construction of Faddeev, Reshetikhin, and Takhtajan.

\section{Universal $R$-Matrices and the Universal FRT-Bialgebra}

In this section, we will construct the universal bialgebra of Faddeev, Reshetikhin, and Takhtajan \cite{FRT88} associated to the braiding on the quantum half-spin representation of ${\mathcal U}_q(\mathfrak{g}_{{\bf I^\prime}})$. We begin with the reduced expression
\[ w_0^{{\bf I^\prime}} = s_6s_5s_4s_3s_6s_5s_4s_6s_5s_6s_2s_4s_5s_6s_3s_4s_5s_2s_4s_3  \]
for the longest element of $W_{{\bf I^\prime}}$. Using Lusztig's action of the braid group on ${\mathcal U}_q(\mathfrak{g}_{{\bf I^\prime}})$ we obtain root vectors (which depend on our choice of reduced expression for $w_0^{{\bf I^\prime}}$) for the algebra ${\mathcal U}_q(\mathfrak{g}_{{\bf I^\prime}})$.  See, for instance, \cite[I.6.7]{BG}, \cite[\textsection 9.1.B]{CP}, or \cite[Ch. 37]{L}. We use the version of the braid group action given by the formulas in \cite[\textsection 2.3]{Y}. For brevity, label the root vector of degree $-e_i+e_j$ ($i\neq j$) obtained from this construction by $E_{ij}$, the root vector of degree $e_i+e_j$ ($i<j$) by $E_{ij}^\prime$, and the root vector of degree $-e_i-e_j$ ($i>j$) by $E_{ij}^\prime$.  The quantized enveloping algebra ${\mathcal U}_q({\mathfrak{g}_{{\bf I^\prime}}})$ is therefore generated by $E_{ij},E_{ij}^\prime$ ($i\neq j$) together with $K_i^{\pm 1}$ ($i\in{\bf I^\prime}$).

Let $V$ be a vector space with basis $\{v_1,...,v_5\}$, and define the quantum exterior algebra
\[\Lambda_q : =  T(V)/\langle v_\ell^2,v_iv_j+qv_jv_i:1\leq j<i\leq 5\text{ and }1\leq \ell\leq 5\rangle,\]
where $T(V)$ is the tensor algebra of $V$. Thus, $\{u_I: I=\{i_1<i_2<\cdots < i_\ell\}\subseteq\{1,...,5\}\}$ is a basis of $\Lambda_q$. Let ${\mathcal S}$ denote the subalgebra of $\Lambda_q$ generated by $\{u_I: I\in{\mathcal B}\}$.

\begin{definition/proposition}\label{halfspin_defn} There is a $k$-algebra homomorphism $\rho:{\mathcal U}_q(\mathfrak{g}_{{\bf I^\prime}})\to\text{End}_k({\mathcal S})$ (called the quantum half-spin representation) defined as follows:
\begin{align*}
&\rho(E_{ij})(u_Iv_j) = (-q)^{i-j-\text{sign}(i-j)}u_Iv_i,   &   &(i\neq j\text{ and }j\notin I),\\
&\rho(E_{ij})(u_I) = 0,                                                        &   &(i\neq j\text{ and }j\notin I),\\
&\rho(E_{ij}^\prime)(u_Iv_iv_j) = (-q)^{c(i,j,I)}u_I,        &   &(i<j\text{ and } i,j\notin I), \\
&\rho(E_{ij}^\prime)(u_I) = (-q)^{c(i,j,I)}u_Iv_jv_i,        &   &(i>j), \\
&\rho(E_{ij}^\prime)(u_I) =0,                                            &   &(i<j\text{ and } \{i,j\}\not\subseteq I),\\
&\rho(K_i)(u_I) = q^{\left<\alpha_i,wt(I)\right>}u_I, & &(i\in {\bf I^\prime}),
\end{align*}
where $c(i,j,I)=sign(i-j)(i+j-3-2|I|)$.
\end{definition/proposition}

The quantum half-spin representation is irreducible with highest weight $\varpi_2$. The vector $u_{\emptyset}$ is a highest weight vector (c.f.  Figure \ref{poset B}). Observe also that the set of weights for the representation ${\mathcal S}$ is exactly $\Phi_w$. In fact, the vector subspace of ${\mathcal U}_q^+[w]$ spanned by the root vectors $Y_I$ ($I\in{\mathcal B}$) is stable under the adjoint action of ${\mathcal U}_q(\mathfrak{g}_{{\bf I^\prime}})$ and we have the following

\begin{figure}
\caption{The Poset ${\mathcal B}$ with the action of the negative root vectors $F_i$ ($i\in{\bf I^\prime}$) indicated by the arrows. The positive root vectors $E_i$ ($i\in{\bf I^\prime}$) act in the opposite direction.}
\label{poset B}
\[\xymatrix{
&&&&[23]\ar@{<-}[r]^3\ar[d]^5&[13]\ar[d]^5&[12]\ar[l]_4&\emptyset\ar[l]_2\\
&&[1234]\ar@{<-}[r]^2\ar[d]^6&[34]\ar@{<-}[r]^4\ar[d]^6&[24]\ar@{<-}[r]^3\ar[d]^6&[14]\ar[d]^6\\
&&[1235]\ar@{<-}[r]^2\ar[d]^5&[35]\ar@{<-}[r]^4\ar[d]^5&[25]\ar@{<-}[r]^3&[15]\\
[2345]\ar@{<-}[r]^3&[1345]\ar@{<-}[r]^4&[1245]\ar@{<-}[r]^2&[45] &&}\]
\end{figure}
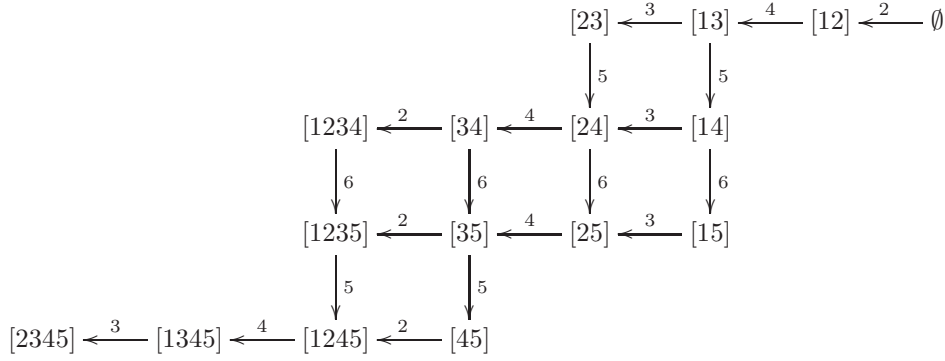

\begin{theorem} There is an isomorphism $\displaystyle{\varphi: \bigoplus_{I\in{\mathcal B}}kY_I\to {\mathcal S}}$ of ${\mathcal U}_q(\mathfrak{g}_{{\bf I^\prime}})$-modules that sends $Y_\emptyset$ to the highest weight vector $u_{\emptyset}$. Under this isomorphism, $Y_I$ maps to $(-q)^{ht(\emptyset)-ht(I)}u_I$, where $ht:{\mathcal B}\to\mathbb{N}$ is a height function on the graded poset ${\mathcal B}$.
\end{theorem}

We now recall the details on constructing the $R$-matrix for ${\mathcal S}$. For every $i,j$ with $1\leq i<j\leq 5$, define $\varphi_{ij}:= \text{exp}_q((1-q^{-2})E_{ij}\otimes E_{ji})$ and $\varphi_{ij}^\prime = \text{exp}_q((1-q^{-2})E_{ij}^\prime\otimes E_{ji}^\prime)$,
where
\[ \text{exp}_q(x) = \sum_{n=0}^\infty\frac{q^{n(n+1)/2}}{[n]_q!}x^n,  \hspace{.6in} [n]_q=\frac{q^n-q^{-n}}{q-q^{-1}},  \hspace{.6in} [n]_q! = [n]_q[n-1]_q\cdots [1]_q,  \]
and let
\begin{equation*}
{\mathcal R}=(\varphi_{45}\varphi_{35}\varphi_{25}\varphi_{15})(\varphi_{34}\varphi_{24}\varphi_{14})(\varphi_{23}\varphi_{13})\varphi_{12}(\varphi_{45}^\prime\varphi_{35}^\prime\varphi_{25}^\prime\varphi_{15}^\prime)(\varphi_{34}^\prime\varphi_{24}^\prime\varphi_{14}^\prime)(\varphi_{23}^\prime\varphi_{13}^\prime)\varphi_{12}^\prime.
\end{equation*}

Since $X^2$ ($X=E_{ij},E_{ij}^\prime$: $i\neq j$) acts trivially on ${\mathcal S}$, only finitely many terms in $(\rho\otimes\rho)({\mathcal R})$ act nontrivially on ${\mathcal S}\otimes{\mathcal S}$. Let $B:{\mathcal S}\otimes{\mathcal S}\to{\mathcal S}\otimes{\mathcal S}$ denote the linear map given by
\begin{equation*}
B(u_I\otimes u_J) = q^{\left<wt(I),wt(J)\right>}u_I\otimes u_J.
\end{equation*}
The universal $R$-matrix associated to ${\mathcal S}\otimes{\mathcal S}$ is
\begin{equation*}
\widehat{R}:=\tau\circ B\circ \left((\rho\otimes\rho)({\mathcal R})\right)\in End_k({\mathcal S}\otimes{\mathcal S}),
\end{equation*}
where $\tau$ is the flip map $\tau(u\otimes v)=v\otimes u$.  One of the most important properties of universal $R$-matrices is given in the following theorem (see, e.g., \cite{CP, KS97}).

\begin{theorem} The universal $R$-matrix $\widehat{R}:{\mathcal S}\otimes{\mathcal S}\to{\mathcal S}\otimes{\mathcal S}$ is an automorphism of ${\mathcal U}_q(\mathfrak{g}_{{\bf I^\prime}})$-modules. Moreover, $\widehat{R}$ satisfies the braid-version of the quantum Yang-Baxter equation
\[ \widehat{R}_{12}\widehat{R}_{23}\widehat{R}_{12} =\widehat{R}_{23}\widehat{R}_{12}\widehat{R}_{23},\]
where $\widehat{R}_{ij}:{\mathcal S}^{\otimes 3}\to{\mathcal S}^{\otimes 3}$ acts via $\widehat{R}$ on the $i$-th and $j$-th tensor components and the identity on the remaining component.
\end{theorem}

\noindent\emph{Remark:}
In \cite{BZ}, Berenstein and Zwicknagl define a family of quadratic algebras called braided symmetric algebras. The quantum Schubert cell algebra ${\mathcal U}_q^+[w]$ falls into this family. In \cite[Thm. 5.4]{Z}, Zwicknagl proves that the quotient of the tensor algebra $T({\mathcal S})$ by the ideal generated by the $-1$-eigenspace of $\widehat{R}$ (which is ${\mathcal U}_q(\mathfrak{g}_{{\bf I^\prime}}).(u_{[12]}\otimes u_{\emptyset}-qu_{\emptyset}\otimes u_{[12]})$ in our case) is isomorphic to ${\mathcal U}_q^+[w]$. Moreover, $Y_I\mapsto [u_I]$ (for $I\in{\mathcal B}$) defines an algebra isomorphism from ${\mathcal U}_q^+[w]$ onto $T(\mathcal S)/\langle{\mathcal U}_q(\mathfrak{g}_{{\bf I^\prime}}).(u_{[12]}\otimes u_{\emptyset}-qu_{\emptyset}\otimes u_{[12]})\rangle$.

We wish to construct a bialgebra ${\mathcal A}$ so that ${\mathcal S}$ is a right ${\mathcal A}$-comodule and $\widehat{R}:{\mathcal S}\otimes{\mathcal S}\to{\mathcal S}\otimes {\mathcal S}$ is an automorphism of right ${\mathcal A}$-comodules. Among all bialgebras satisfying these conditions, we want ${\mathcal A}$ to satisfy a universal property with respect to this condition. More precisely, we have the following definition and proposition, due to Faddeev, Reshetikhin, and Takhtajan (for details see, e.g., \cite[\textsection  9.1-9.2]{KS97}).

\begin{definition}[The FRT-Bialgebra ${\mathcal A}$]
If
\begin{equation}
\widehat{R}(u_I\otimes u_J)= \sum_{K,L\in{\mathcal B}} R_{IJ}^{KL}u_L\otimes u_K,
\end{equation}
where $R_{IJ}^{KL}\in k$ and $I,J,K,L\in{\mathcal B}$, then
the \emph{FRT-bialgebra} ${\mathcal A}$ is generated by $\{X_{IJ}:I,J\in{\mathcal B}\}$ and has defining relations
\begin{equation}
\label{FRT-definition}\sum_{K,L\in{\mathcal B}} R_{KL}^{TS}X_{KI}X_{LJ}=\sum_{K,L\in{\mathcal B}}R_{IJ}^{KL}X_{SL}X_{TK}
\end{equation}
for all $S,T,I,J\in{\mathcal B}$. The comultiplication and counit are defined by $\Delta_{\mathcal A}(X_{IJ})= \sum_{K\in{\mathcal B}} X_{IK}\otimes X_{KJ}$ and $\epsilon_{\mathcal A}(X_{IJ})=\delta_{IJ}$.
\end{definition}

The following can be found in \cite[Section 9.1.1, Proposition 2]{KS97}.

\begin{proposition} The linear map $\rho_{{\mathcal A}}:{\mathcal S}\to{\mathcal S}\otimes{\mathcal A}$ defined by $\rho_{\mathcal A}(u_I)=\sum_{J\in{\mathcal B}}u_J\otimes X_{JI}$ gives ${\mathcal S}$ the structure of a right ${\mathcal A}$-comodule so that $\widehat{R}:{\mathcal S}\otimes{\mathcal S}\to{\mathcal S}\otimes{\mathcal S}$ is a homomorphism of right ${\mathcal A}$-comodules. If ${\mathcal A}^\prime$ is another bialgebra such that ${\mathcal S}$ is a right ${\mathcal A}^\prime$-comodule (with structure map $\rho_{{\mathcal A}^\prime}:{\mathcal S}\to{\mathcal S}\otimes {\mathcal A}^\prime$) and $\widehat{R}$ a homomorphism of right ${\mathcal A}^\prime$-comodules, then there is a unique homomorphism of bialgebras $\psi:{\mathcal A}\to{\mathcal A}^\prime$ such that $(id\otimes\psi)\circ\rho_{\mathcal A} = \rho_{{\mathcal A}^\prime}$.
\end{proposition}

The partial ordering $\preceq$ on ${\mathcal B}\times{\mathcal B}$ defined in Section~\ref{section quantum schubert cells} is useful for determining when certain coefficients of $\widehat{R}$ are nonzero. Observe first that $\widehat{R}$ is homogeneous with respect to the $Q$-grading on ${\mathcal S}\otimes{\mathcal S}$. Thus, if $R_{IJ}^{KL}$ is nonzero, then $wt(I)+wt(J)=wt(K)+wt(L)$. Secondly, each $\varphi_{ij}$ (and $\varphi_{ij}^\prime$) acts on a pure tensor $u_I\otimes u_J\in{\mathcal S}\otimes{\mathcal S}$ by either sending it to itself or to something of the form $u_I\otimes u_J+q^{-1}\hat{q}u_K\otimes u_L$, with $I<K$. Thus, we have the following proposition.
\begin{proposition}
If $R_{IJ}^{KL}\in k$ is nonzero, then  $(I,J)\preceq (K,L)$.
\end{proposition}

The formula giving the coefficient $R_{IJ}^{KL}\in k$ of the $R$-matrix depends on the size of the equivalence class of $(I,J)$ as well as the difference in heights of $(I,J)$ and $(K,L)$ with respect to the partial ordering $\preceq$ on ${\mathcal B}\times{\mathcal B}$. The following theorem can be verified by checking each case listed below.

\begin{theorem}\label{R-matrix coefficients}
If  $(I,J)\preceq (K,L)$, then
\begin{equation}
R_{IJ}^{KL} = \begin{cases} q^{\left<wt(I),wt(J)\right>},             &  (K,L)=(I,J),\\
\hat{q}(q-(-q)^{ht(I,J)-ht(K,L)+1}), & (K,L)=(J,I)\text{ and } |[(I,J)]|=8, \\
\hat{q}q, & (K,L)=(J,I)\text{ and } |[(I,J)]|=2, \\
\hat{q}(-q)^{ht(I,J)-ht(K,L)+1}, & otherwise. \end{cases}
\end{equation}
\end{theorem}

\section{Main Results}\label{section_mainresults}

The following theorem gives the defining relations among any ``row" of generators of the FRT-bialgebra.

\begin{theorem}\label{FRT relations1}For every $S\in{\mathcal B}$, the subalgebra ${\mathcal A}_S\subseteq{\mathcal A}$ generated by $\{X_{SI}:I\in{\mathcal B}\}$ has the following defining relations:
\begin{enumerate}[(i)]
\item For every $I\not\geq J$,
\begin{equation}
X_{SI}X_{SJ} =q^{\left<wt(I),wt(J)\right>}X_{SJ}X_{SI}+\hat{q}\!\!\!\!\!\!\!\sum_{\tiny\begin{array}{c}(I,J)\prec (L,M),\\ M\leq_{lexico.}L\end{array}}\!\!\!\!\!\!\!(-q)^{ht(L,M)-ht(I,J)-1}X_{SL}X_{SM}.
\end{equation}
\item For every $1\leq i\leq 10$,
\begin{equation}
\label{10extra}X_{S,a_{i1}}X_{S,A_{i1}}-qX_{S,a_{i2}}X_{S,A_{i2}}+q^2X_{S,a_{i3}}X_{S,A_{i3}}-q^3X_{S,a_{i4}}X_{S,A_{i4}}=0,
\end{equation}
where  $(a_{ij},A_{ij})\in{\mathcal B}\times{\mathcal B}$ is in $i$-th row and $j$-th column of Table \ref{large classes}.
\end{enumerate}
\end{theorem}

\begin{proof}
In the defining relations of Eqn. \ref{FRT-definition}, the only instance where a product of variables, $X_{SI}X_{SJ}$, from the same ``row" appears (with nonzero coefficients) are in the equations
\[q^2X_{SI}X_{SJ}=\sum_{L,M\in{\mathcal B}}R_{IJ}^{ML}X_{SL}X_{SM}.\] The types of relations involving a pair of variables, $X_{SI}$ and $X_{SJ}$, depends on the size of the equivalence class of $(I,J)$, which could be $1$, $2$, or $8$.  We consider these three cases separately. First of all, if $|[(I,J)]|=1$ then $I=J$ and the only equation obtained involving the product $X_{SI}X_{SJ}$ ($=X_{SJ}^2$) is $q^2X_{SJ}^2=q^2X_{SJ}^2$, which is superfluous. On the other hand, if $I$ and $J$ are chosen so that $|[(I,J)]|=2$, then $I$ and $J$ are comparable. Thus, without loss of generality suppose $I<J$. In this case, the only relations involving the variables $X_{SI}$ and $X_{SJ}$ are $q^2X_{SI}X_{SJ}=qX_{SJ}X_{SI}+q\hat{q}X_{SI}X_{SJ}$ and $q^2X_{SJ}X_{SI}=qX_{SI}X_{SJ}$. However, these two equations are equivalent to the single identity $X_{SI}X_{SJ}=qX_{SJ}X_{SI}$. Finally, if $I$ and $J$ are chosen so that $|[(I,J)]|=8$, then there are potentially up to eight defining relations involving the variables $X_{SI}$ and $X_{SJ}$. Let $\{(a_1,A_1),...,(a_4,A_4),(A_1,a_1),...,(A_4,a_4)\}$ denote of equivalence class of $(I,J)$, where $(a_i,A_i)$ ($i=1,...,4$) has height $i$. Thus, $(I,J)$ is among one of the eight elements listed. There are eight defining relations involving the variables $X_{S,a_1},..,X_{S,a_4},X_{S,A_1},...,X_{S,A_4}$, which can be summarized by the matrix equation $(A-SkewDiag(q^2)){\bf X}^t=0$, where $SkewDiag(q^2)$ is the $8\times 8$ matrix with entry $q^2$ everywhere along the skew-diagonal and $0$'s elsewhere,
\[
{\bf X}:=(X_{S,A_1}X_{S,a_1},...,X_{S,A_4}X_{S,a_4},X_{S,a_4}X_{S,A_4},...,X_{S,a_1}X_{S,A_1}),
\]
and
\begin{equation*}
A:=\left(\begin{array}{cccccccc} 1& \hat{q}& -\hat{q}q^{-1}& \hat{q}q^{-2}& \hat{q}q^{-2}& -\hat{q}q^{-3}& \hat{q}q^{-4}&   \hat{q}(q-q^{-5})   \\
0& 1& \hat{q}& -\hat{q}q^{-1}& -\hat{q}q^{-1}& \hat{q}q^{-2}&   \hat{q}(q-q^{-3}) & \hat{q}q^{-4} \\
0& 0& 1& \hat{q}& \hat{q}&   \hat{q}^2 & \hat{q}q^{-2}& -\hat{q}q^{-3}\\
0& 0& 0& 1& 0& 1& -\hat{q}q^{-1}& \hat{q}q^{-2}\\
0& 0& 0& 0& 1& \hat{q}& -\hat{q}q^{-1}&\hat{q}q^{-2}\\
0& 0& 0& 0&0& 1& \hat{q}& -\hat{q}q^{-1}\\
0& 0& 0& 0& 0& 0& 1& \hat{q}\\
0& 0& 0& 0& 0& 0& 0& 1
\end{array}\right)\!\!\! \begin{array}{c}\\ \\ \\ \\ \\ \\ \\ \\ \cdot\end{array}
\end{equation*}
The entries in $A$ are obtained by reading off the appropriate coefficients of $\widehat{R}$, which are given in Thm. \ref{R-matrix coefficients}. The matrix $A-SkewDiag(q^2)$ has rank five. Hence, the system of eight equations given by $(A-SkewDiag(q^2)){\bf X}^t=0$ will reduce to an equivalent system of five equations. In our case, the five equations are
\begin{align*}
&X_{S,a_1}X_{S,A_1}=X_{S,A_1}X_{S,a_1} +\hat{q}(X_{S,a_2}X_{S,A_2}-qX_{S,a_3}X_{S,A_3}+q^2X_{S,a_4}X_{S,A_4}),\\
&X_{S,a_2}X_{S,A_2}=X_{S,A_2}X_{S,a_2}+\hat{q}(X_{S,a_3}X_{S,A_3}-qX_{S,a_4}X_{S,A_4}),\\
&X_{S,a_3}X_{S,A_3}=X_{S,A_3}X_{S,a_3}+\hat{q}X_{S,a_4}X_{S,A_4},\\
&X_{S,a_4}X_{S,A_4}=X_{S,A_4}X_{S,a_4},\\
&X_{S,a_1}X_{S,A_1}-qX_{S,a_2}X_{S,A_2}+q^2X_{S,a_3}X_{S,A_3}-q^3X_{S,a_4}X_{S,A_4}=0.
\end{align*}
The relations arising from the three cases considered account for all commutation relations and can be succinctly written as they are in the statement of this theorem.
\end{proof}

Our first main result is the following corollary of Theorem \ref{FRT relations1}. Each ``row" of ${\mathcal A}$ is isomorphic to a quotient of ${\mathcal U}_q^+[w]$.

\begin{corollary}
For every $S\in{\mathcal B}$, there is a surjective algebra homomorphism
\begin{align*}
\psi_S:  {\mathcal U}_q^+[w]    &\longrightarrow {\mathcal A}_S \\
 Y_I  &\longmapsto X_{SI} ,
\end{align*}
with kernel $\left<ad({\mathcal U}_q(\mathfrak{g}_{{\bf I^\prime}}))\Theta\right>$.
\end{corollary}

\begin{proof} The defining relations of ${\mathcal U}_q^+[w]$ and ${\mathcal A}_S$ show that $\psi_S$ defines a surjective homomorphism. The ten extra defining relations of ${\mathcal A}_S$ given in Eqn. \ref{10extra} imply that the ideal of ${\mathcal U}_q^+[w]$ generated by the ten-dimensional vector space $ad({\mathcal U}_q(\mathfrak{g}_{{\bf I^\prime}}))\Theta$ is the kernel of $\psi_S$.
\end{proof}

The following theorem gives the defining relations among certain pairs of rows of generators of ${\mathcal A}$.

\begin{theorem}\label{FRT relations2}For every $S,T\in{\mathcal B}$ such that $|(S\cup T)\backslash (S\cap T)|=2$ and $S<T$, the subalgebra ${\mathcal A}_{(S,T)}\subseteq{\mathcal A}$ generated by $\{X_{SI},X_{TI}:I\in{\mathcal B}\}$ has the following defining relations.
\begin{enumerate}[(i)]
\item Commutation relations among variables in the same row: For every $I,J\in{\mathcal B}$ such that $I\not\geq J$,
\begin{align}
&\label{row1 relations}X_{SI}X_{SJ} =q^{\left<wt(I),wt(J)\right>}X_{SJ}X_{SI}+\hat{q}\!\!\!\!\!\!\!\sum_{\tiny\begin{array}{c}(I,J)\prec (L,M),\\ M\leq_{lexico.}L\end{array}}\!\!\!\!\!\!\!(-q)^{ht(L,M)-ht(I,J)-1}X_{SL}X_{SM},\\
&\label{row2 realtions}X_{TI}X_{TJ} =q^{\left<wt(I),wt(J)\right>}X_{TJ}X_{TI}+\hat{q}\!\!\!\!\!\!\!\sum_{\tiny\begin{array}{c}(I,J)\prec (L,M),\\ M\leq_{lexico.}L\end{array}}\!\!\!\!\!\!\!(-q)^{ht(L,M)-ht(I,J)-1}X_{TL}X_{TM}.
\end{align}
\item Commutation relations among variables in different rows:  For every $I,J\in{\mathcal B}$,
\begin{equation}\label{tworows relations}
X_{SI}X_{TJ} =q^{\langle wt(I),wt(J)\rangle-1}X_{TJ}X_{SI}+\hat{q}\left(\sum_{(I,J)\prec (L,M)}\!\!\!\!\!\!\!(-q)^{ht(L,M)-ht(I,J)-1}X_{SL}X_{TM}+q^{-1}\epsilon(I,J)X_{SJ}X_{TI}\right),
\end{equation}
where $\epsilon(I,J)$ is defined in Eqn. \ref{epsilondefn}.
\item Other relations: For every $1\leq i\leq 10$,
\begin{align}
&\label{other1 relations}X_{S,a_{i1}}X_{S,A_{i1}}-qX_{S,a_{i2}}X_{S,A_{i2}}+q^2X_{S,a_{i3}}X_{S,A_{i3}}-q^3X_{S,a_{i4}}X_{S,A_{i4}}=0,\\
&\label{other2 relations}X_{T,a_{i1}}X_{T,A_{i1}}-qX_{T,a_{i2}}X_{T,A_{i2}}+q^2X_{T,a_{i3}}X_{T,A_{i3}}-q^3X_{T,a_{i4}}X_{T,A_{i4}}=0,\\
&\label{other3 relations}\sum_{1\leq j\leq 8}(-q)^{ht(a_{ij},A_{ij})}X_{S,a_{ij}}X_{T,A_{ij}}=0,
\end{align}
where  $(a_{ij},A_{ij})\in{\mathcal B}\times{\mathcal B}$ is in $i$-th row and $j$-th column of Table \ref{large classes}.
\end{enumerate}
\end{theorem}

\begin{proof} Theorem \ref{FRT relations1} gives the defining relations among any row of generators. Hence, it remains to verify the relations involving a product of a variable in the $S$-th row with a variable in the $T$-th row. The only instances such a product occurs with nonzero coefficient are in the equations
\begin{equation}\label{eqn1}q\hat{q}X_{SI}X_{TJ}+qX_{TI}X_{SJ}=\sum_{K,L} R_{IJ}^{KL}X_{SL}X_{TK} \end{equation}
and
\begin{equation}\label{eqn2} qX_{SI}X_{TJ}=\sum_{K,L}R_{IJ}^{KL}X_{TL}X_{SK} \end{equation}
for $I,J\in{\mathcal B}$. These equations simplify differently based on the size of the symmetric difference $(I\cup J)\backslash (I\cap J)$, which could be $0$, $2$, or $4$. We will consider these three cases separately.

First, if $|(I\cup J)\backslash (I\cap J)|=0$, then $I=J$, and the above equations simplify to $q\hat{q}X_{SI}X_{TI}+qX_{TI}X_{SI}=q^2X_{SI}X_{TI}$ and $qX_{SI}X_{TI}=q^2X_{TI}X_{SI}$. However, one of these equations is superfluous. In fact they reduce to the single relation $X_{SI}X_{TI}=qX_{TI}X_{SI}$.

If $|(I\cup J)\backslash (I\cap J)|=2$, then $I$ and $J$ are comparable. Without loss of generality suppose $I<J$. Thus,  Eqns. \ref{eqn1}  and \ref{eqn2} simplify to $q\hat{q}X_{SI}X_{TJ}+qX_{TI}X_{SJ}=qX_{SJ}X_{TI}+q\hat{q}X_{SI}X_{TJ}$ and $qX_{SI}X_{TJ}=qX_{TJ}X_{SI}+q\hat{q}X_{TI}X_{SJ}$ respectively. These in turn reduce to $X_{SJ}X_{TI}=X_{TI}X_{SJ}$ and $X_{SI}X_{TJ}=X_{TJ}X_{SI}+\hat{q}X_{SJ}X_{TI}$.

Finally, if $|(I\cup J)\backslash (I\cap J)|=4$, then $|[(I,J)]|=8$. Let $\{(a_1,A_1),...,(a_4,A_4),(A_1,a_1),...,(A_4,a_4)\}$ denote of equivalence class of $(I,J)$, where $(a_i,A_i)$ ($i=1,...,4$) has height $i$.  There are 16 defining relations involving of product of a variable from the $S$-th row, $X_{S,a_1},..,X_{S,a_4},X_{S,A_1},...,X_{S,A_4}$ with a variable from the $T$-th row, $X_{T,a_1},..,X_{T,a_4},X_{T,A_1},...,X_{T,A_4}$.  Half of these 16 defining relations come from Eqn. \ref{eqn1} and the other half come from Eqn. \ref{eqn2}. These 16 equations are summarized by the matrix equation
\[ \left(\begin{array}{cc} A-SkewDiag(q\hat{q}) & SkewDiag(-q) \\ SkewDiag(-q) & A \end{array}\right) ({\bf X_{ST}})^t =0, \]
where $SkewDiag(x)$  is the $8\times 8$ matrix with entry $x$ everywhere along the skew-diagonal and $0$'s elsewhere,
\[\resizebox{1\hsize}{!}{$ {\bf X_{ST}}:=(X_{S,A_1}X_{T,a_1},...,X_{S,A_4}X_{T,a_4},X_{S,A_4}X_{T,a_4},...,X_{S,A_1}X_{T,a_1},X_{T,A_1}X_{S,a_1},...,X_{T,A_4}X_{S,a_4},X_{T,a_4}X_{S,A_4},...,X_{T,a_1}X_{S,A_1}),$} \]
and $A$ is the same matrix in the proof of Theorem \ref{FRT relations1}. The matrix $\left(\begin{array}{cc} A-SkewDiag(q\hat{q}) & SkewDiag(-q) \\ SkewDiag(-q) & A \end{array}\right)$ has rank $9$. Therefore the system of 16 equations reduces an equivalent system of 9 equations. In our case, they reduce to the following:
\begin{enumerate}
\item $X_{T,a_1}X_{S,A_1} = qX_{S,A_1}X_{T,a_1}-qf_{ST}(A_1,a_1)$,
\item $X_{T,a_2}X_{S,A_2} = qX_{S,A_2}X_{T,a_2}-qf_{ST}(A_2,a_2)$,
\item $X_{T,a_3}X_{S,A_3} = qX_{S,A_3}X_{T,a_3}-qf_{ST}(A_3,a_3)$,
\item $X_{T,a_4}X_{S,A_4} = qX_{S,A_4}X_{T,a_4}-qf_{ST}(A_4,a_4)$,
\item $X_{T,A_4}X_{S,a_4} = qX_{S,a_4}X_{T,A_4}-qf_{ST}(a_4,A_4)$,
\item $X_{T,A_3}X_{S,a_3} = qX_{S,a_3}X_{T,A_3}-qf_{ST}(a_3,A_3)-\hat{q}X_{S,A_3}X_{T,a_3}$,
\item $X_{T,A_2}X_{S,a_2} = qX_{S,a_2}X_{T,A_2}-qf_{ST}(a_2,A_2)-\hat{q}X_{S,A_2}X_{T,a_2}$,
\item $X_{T,A_1}X_{S,a_1} = qX_{S,a_1}X_{T,A_1}-qf_{ST}(a_1,A_1)-\hat{q}X_{S,A_1}X_{T,a_1}$,
\item $\sum_{(x,y)\sim (I,J)} (-q)^{ht(x,y)}(x,y)=0$,
\end{enumerate}
where $f_{ST}(x,y):=\hat{q}\sum_{(x,y)\prec (L,M)}(-q)^{ht(L,M)-ht(x,y)-1}X_{SL}X_{TM}$.

All of the relations arising from the three cases just considered can be succinctly written as the relations stated in the theorem.
\end{proof}

The algebras ${\mathcal U}_q^+[\widehat{w}]$ and ${\mathcal A}_{(S,T)}$ apparently have very similar defining relations. We remark that the function on generators given by $Z_I\mapsto X_{SI}$, $Z_{I+\delta}\mapsto X_{TI}$, for $I\in{\mathcal B}$, does not lift to an algebra homomorphism ${\mathcal U}_q^+[\widehat{w}]\to{\mathcal A}_{(S,T)}$. However, if we replaced $X_{TJ}X_{SI}$ with $qX_{TJ}X_{SI}$ in Eqn. \ref{tworows relations}, this would indeed define an algebra homomorphism. In other words, \emph{twisting} the multiplication in ${\mathcal A}_{(S,T)}$ (or ${\mathcal U}_q^+[\widehat{w}]$) will yield a homomorphism.

We will use the $\mathbb{Z}\{\alpha_i\}_{i\in{\bf I_0}}$-grading on  ${\mathcal U}_q^+[\widehat{w}]$ to construct an algebra $({\mathcal U}_q^+[\widehat{w}])^\prime$ related to ${\mathcal U}_q^+[\widehat{w}]$, but with a slightly different multiplicative structure. Let $c:\mathbb{Z}\{\alpha_i\}_{i\in{\bf I_0}}\times\mathbb{Z}\{\alpha_i\}_{i\in{\bf I_0}}\to k$ be the bicharacter defined by
\begin{align}
\label{the-cocycle}c(\alpha_i,\alpha_j)=\begin{cases} q, & (i,j)=(0,1),\\ 1, &\text{otherwise}.\end{cases}
\end{align}
As a set $({\mathcal U}_q^+[\widehat{w}])^\prime=\{u^\prime:u\in{\mathcal U}_q^+[\widehat{w}]\}$.
We require the map $\zeta:{\mathcal U}_q^+[\widehat{w}]\to({\mathcal U}_q^+[\widehat{w}])^\prime$ defined by $u\mapsto u^\prime$ (for $u\in{\mathcal U}_q^+[\widehat{w}]$) to be an isomorphism of $k$-modules. This map induces a $\mathbb{Z}\{\alpha_i\}_{i\in{\bf I_0}}$-grading on $({\mathcal U}_q^+[\widehat{w}])^\prime$. Multiplication of homogeneous elements in $({\mathcal U}_q^+[\widehat{w}])^\prime$ is given by
\[  x^\prime y^\prime = c(\text{deg}(x),\text{deg}(y))(xy)^\prime.  \]

Our second main result is the following theorem which gives a description of certain pairs of rows of ${\mathcal A}$. Recall $\Omega_1,\dots,\Omega_{13}$ from Table~\ref{HWVs table}.

\begin{theorem}\label{surjective homomorphism} For every $S,T\in{\mathcal B}$ with $|(S\cup T)\backslash (S\cap T)|=2$ and $S<T$, there is a surjective algebra homomorphism $\psi_{(S,T)}:({\mathcal U}_q^+[\widehat{w}])^\prime\to{\mathcal A}_{(S,T)}$ given by
\begin{align*}
&(Z_I)^\prime\longmapsto X_{SI}, \\
&(Z_{I+\delta})^\prime\longmapsto X_{TI}.
\end{align*}
The kernel of $\psi_{(S,T)}$ is the ideal generated by three $10$-dimensional vector spaces:  $\left(ad({\mathcal U}_q(\mathfrak{g}_{{\bf I^\prime}}))\Omega_j\right)^\prime$ for $j=3,4,5$.
\end{theorem}

\end{document}